\documentclass[a4paper,11pt]{article}
\usepackage[pagewise]{lineno}
\usepackage{amsmath}
\usepackage{amssymb}
\usepackage{mathrsfs}
\usepackage{amsfonts}
\usepackage{amsthm}
\usepackage{pstricks, pst-node, pst-text, pst-3d,psfrag}
\usepackage{graphicx}
\usepackage{indentfirst}
\usepackage{enumerate}
\usepackage{subfigure}
\usepackage{cite}
\usepackage[colorlinks=true]{hyperref}
\usepackage{authblk}
\hypersetup{urlcolor=blue, citecolor=red}

\theoremstyle{plain}
\newtheorem{thm}{Theorem}[section]
\newtheorem{prop}[thm]{Proposition}
\newtheorem{cor}[thm]{Corollary}
\newtheorem{lem}[thm]{Lemma}

\theoremstyle{definition}
\newtheorem{defn}[thm]{Definition}

\newtheorem{rmk}[thm]{Remark}

\numberwithin{equation}{section}

\makeatletter 
\@addtoreset{equation}{section}
\makeatother  

\begin{document}
\title{Prevalent behavior and almost sure Poincar\'{e}-Bendixson Theorem \\ for smooth flows with invariant $k$-cones\thanks{Supported by NSF of China No.11825106, 11971232 and 12090012.}}
\author{Yi Wang}
\author{Jinxiang Yao}
\author{Yufeng Zhang\thanks{Corresponding author: zyfp@mail.ustc.edu.cn (Y. Zhang).}}
\affil{School of Mathematical Science\par University of Science and Technology of China\par Hefei, Anhui, 230026, P. R. China}
\date{}
  \maketitle
\begin{abstract}
We investigate the global dynamics from a measure-theoretic perspective for smooth flows with invariant cones of rank $k$. For such systems, it is shown that prevalent (or equivalently, almost all) orbits will be pseudo-ordered or convergent to equilibria. This reduces to Hirsch's prevalent convergence Theorem if the rank $k=1$; and implies an almost-sure Poincar\'{e}-Bendixson Theorem for the case $k=2$. These results are then applied to obtain an almost sure Poincar\'e-Bendixson theorem for high-dimensional differential equations.
\end{abstract}

\section{Introduction and Main Results}
In this article,  we are interested in the global dynamics {\it in measure-theoretic sense} for flows with invariant cones of rank $k$ (abbr. $k$-cone). Such a flow is also called monotone with respect to $k$-cone $C$ (see Section \ref{sec2}). Here, a $k$-cone $C\subset \mathbb{R}^{n}$ means a closed subset that contains a linear $k$-dimensional subspace and no linear subspaces of higher dimension.

From a measure-theoretic (or probabilistic) perspective, prevalence is a quite useful notion, introduced by Christensen \cite{C72} and Hunt et al.\cite{HSY93}, that describes properties of interest occurs for ``almost surely".  For Euclidean spaces, it is equivalent to notion of full Lebesgue measure (see Section \ref{sec2}). Over decades since its development, prevalence have undergone extensive investigations. We refer to \cite{HSY93,K97,BV10,EN20} (and references therein) for more details.

In dynamical systems, prevalence is parallel to another classical notion called genericity, which formulates in the topological sense the properties of interest occurs residually (i.e., on a countable intersection of open dense subsets in a Baire space). Many examples in dynamical systems are known to be both generic in the topological sense and prevalent in measure-theoretic sense. However, on the other hand, there are also many cases properties that are generic are not prevalent, and vise-versa (see, e.g. \cite{BV10,K97}).

We point out that a $k$-cone is not a standard cone defined in the literature, but adopted from by Fusco-Oliva \cite{FO88,FO91} (see also Krasnosel\'skii et al. \cite{KLS89} for a Banach space). In particular, a convex cone $K$ (defined in the standard sense) gives rise to a $1$-cone, $K\cup(-K)$. Therefore, the class of flows we consider here naturally includes the classical monotone dynamical systems since the ground-breaking work of M.W. Hirsch \cite{H84,H88,H82,H85} (see monographs or surveys \cite{AS03,HS05,P02,H95,H17}, we call it as classical monotone systems). There is however an essential difference between a $k(\ge 2)$-cone and a convex cone $K$, due to the lack of convexity in the former. The lack of convexity requires substantially new ideas for exploring the
implication of ``monotonicity" for the dynamics of the considered flows. Typical examples of flows monotone with respect to $k(\ge 2)$-cones are the high-dimensional competitive systems (see, for example \cite{B13,B19,H88n,JMW09,M20,M94,MNR19,JN17,WJ01}); systems with quadratic cones and Lyapunov-like functions (see, for example, \cite{OS00,R80,R87}); and monotone cyclic feedback systems with negative feedback (\cite{MN13,MS96b,MH90}) arising from a wide range of neural and physiological control systems, etc..

Classical monotone dynamical systems are abundant and important sources of topological genericity and prevalence. For continuous-time systems, the celebrated Hirsch's generic convergence theorem first indicates that the generic precompact orbit approaches the set of equilibria (see Pol\'{a}\v{c}ik \cite{P89,P90} and Smith-Thieme \cite{ST91} for the improved version). Meanwhile, Enciso, Hirsch and Smith \cite{GMH08} investigated the prevalent behavior and proved that the set of points that converge to equilibria is prevalent. For discrete-time systems (mappings), Pol\'a\v{c}ik-Tere\v{s}\v{c}\'{a}k \cite{PT91} first proved that the generic convergence to periodic orbits occurs provided that the mapping $F$ is of class $C^{1,\alpha}$ (see Tere\v{s}\v{c}\'{a}k \cite{T94p} and Wang-Yao \cite{WY20} for $C^1$-smooth systems). Very recently, the present authors \cite{WYZ22} proved the set of points that converge to periodic orbits is prevalent.

An insightful observation for classical strongly monotone systems exhibits a kind of important typical nontrivial orbits, called pseudo-ordered orbits. More generally, for a $k$-cone $C$, a nontrivial orbit $O(x):=\{\Phi_t(x) :t \ge 0\}$ is pseudo-ordered if $O(x)$ possesses one pair of distinct points $\Phi_{t}(x)$, $\Phi_{s}(x)$ such that $\Phi_{t}(x)-\Phi_{s}(x)\in C$. This is also first handled by Sanchez \cite{S09}. Very recently, Feng, Wang and Wu \cite{FWW19} proved that generic (open and dense) orbits in $\mathbb{R}^n$ are either pseudo-ordered or convergent to equilibria. This covers the Hirsch's Generic Convergence Theorem in the case $k=1$. Together with their previous work \cite{FWW17}, a {\it generic Poincar\'{e}-Bendixson Theorem} is thus obtained for the case $k=2$ (see \cite[Theorem B]{FWW19}).

In the present paper, we focus on the global behavior of a flow in strongly monotone
with respect to $k(\ge 2)$-cone $C$ {\it from a measure-theoretic perspective}. Before describing our main results and our approach, we formulate the standard assumptions:

\vskip 2mm

\begin{description}
  \item[(H)] The flow $\Phi_{t}$ is $C^{1,\alpha}$-smooth and strongly monotone with respect to a $k$-solid cone $C$ in $\mathbb{R}^{n}$. Moreover, the $x$-derivative $D_{x}\Phi_{t}$ satisfies $D_{x}\Phi_{t}(C\backslash\{0\})\subset {\rm Int}C$ for $t>0$.
\end{description}

\noindent We refer to Section \ref{sec2} for more detailed definitions about strong monotonicity and others. Throughout the paper, we always assume {\bf (H)} holds. Let
$$ Q=\{x\in \mathbb{R}^n:\text{ the orbit } O(x) \text{ is pseudo-ordered}\}.$$

\vspace{2 ex}

\noindent$\textbf{Theorem A (Prevalent Dynamics Theorem)}.$
{\it Let $D\subset \mathbb{R}^n$ be an open bounded set such that the orbit set $O(D)$ of $D$ is bounded. Then the set $\{x\in D: x\in Q \text{ or } \omega(x)  \text{ is a singleton}\}$ is prevalent in $D$. Furthermore, the set has full Lebesgue measure in $D$.}

\vspace{2 ex}

Theorem A concludes that, for almost every point in $D$, the orbit is either pseudo-ordered or convergent to a single equilibrium. As we mentioned above, the set $$D^\prime\triangleq\{x\in D: x\in Q \text{ or } \omega(x)  \text{ is a singleton}\}$$ is also {\it generic} (i.e., contains an open and dense subset) in $D$ (c.f. \cite[Theorem A]{FWW19}). This entails that, for smooth flow strongly monotone
with respect to $k$-cone, the fact ``convergence to an equilibrium or the orbit is pseudo-ordered" is both prevalent in measure-theoretic sense and generic in the topological sense.

Needless to say, if the rank $k=1$, Theorem A in conjunction with the Monotone Convergence Criterion (see \cite{H95}) naturally reduces to the prevalent convergence obtained by Enciso, Hirsch and Smith \cite[Theorem 1]{GMH08}
for $\mathbb{R}^n$.

\vspace{1 ex}
Moreover, if the rank $k=2$, we have the following:

\vspace{1 ex}

\noindent$\textbf{Theorem B (Almost Sure Poincar\'e-Bendixson Theorem)}.$
{\it Let all hypotheses in Theorem A hold. Assume also that $k=2$ and $C$ is completed. Then, for almost every point $x\in D$, the $\omega$-limit set $\omega(x)$ containing no equilibria is a single closed orbit.}

\vspace{1 ex}
As mentioned above, Feng et.al \cite{FWW19} obtained the so called ``{\it Generic Poincar\'e-Bendixson Theorem}", that is, for generic (open and dense) points $x\in D$, the $\omega$-limit set $\omega(x)$ containing no equilibria is a single closed orbit. A drawback of topological genericity is that open dense subsets can be arbitrarily small in terms of measure. However, Theorem B here exhibits that, for smooth flow strongly monotone with respect to $2$-cone, the ``Poincar\'e-Bendixson Theorem" holds not only generically in the topological sense, but also almost-surely in measure-theoretic sense.

Compared to the classical monotone dynamical flows, the generic/prevalent behavior of the flows with invariant $k(\ge 2)$-cones is much more complicated. Due to the loss of convexity, the ``order"-relation defined by $C$ (see Section \ref{sec2}) is not a partial order, because it is neither antisymmetric nor transitive. So, it requires substantially new techniques for exploring the implication of monotonicity for the dynamics of the considered flows.

Our approach involves the ergodic argument using the $k$-exponential separation and the associated $k$-Lyapunov exponent (that reduces to the first Lyapunov exponent if $k=1$), which was first introduced in \cite{FWW19}. Plus, we present a novel interesting lemma, called {\it the Probe Lemma} (see Section \ref{sec-DL}), which plays a decisive role for the proof of the main results. Roughly speaking, this lemma indicates that one can detect the pseudo-ordered orbits surrounding the point $z\notin D^\prime$, by tracking the information restricted on a $k$-dimensional probe $P$ passing through $z$ (see Lemma \ref{probe}). Based on the Probe Lemma, we succeeded in proving the main results in Section \ref{sec4}.

In section \ref{sec6}, our main results will be applied to an $n$-dimensional ODE system with a quadratic cone and obtain the almost sure Poincar\'e-Bendixson theorem in such  high-dimensional system. R. A. Smith \cite{R80,R87} has ever obtained a Poincar\'e-Bendixson theorem for such system under the assumption of the existence of  a quadratic Lyapunov function. The connections between Smith's results and the systems with invariant $2$-cones was observed by Sanchez \cite{S09}. Feng et.al \cite{FWW19} proved a generic Poincar\'e-Bendixson Theorem even if a quadratic Lyapunov function can not be constructed. Our Theorem B here shows that for such system, the ``Poincar\'e-Bendixson Theorem" holds not only generically, but almost-surely in measure-theoretic sense.

\noindent\section{Notations and Preliminary Results}\label{sec2}
In this section, we will fix some fundamental notations. Let $\mathbb{R}^{n}$ be the $n$-dimensional Euclidean space with a norm $\|\cdot\|$ and $\Phi(t,x)$ be the \emph{flow} on $\mathbb{R}^{n}$, which is a continuous map $\Phi: \mathbb{R}\times\mathbb{R}^{n}\rightarrow\mathbb{R}^{n}$ with: $\Phi_{0}(x)=(x)$ and $\Phi_{t}\circ\Phi_{s}(x)=\Phi_{t+s}(x)$ for $t,s\in \mathbb{R}$. We will henceforth continue to denote $\Phi(t,x)$ by $\Phi_{t}(x)$. A \emph{$C^{1,\alpha}$-smooth} flow $\Phi_{t}$ on $\mathbb{R}^{n}$ is a flow which $\Phi|_{\mathbb{R}\times\mathbb{R}^{n}}$ is a $C^{1,\alpha}$-map (a $C^{1}$-map with a locally $\alpha$-H\"older derivative) with $\alpha\in(0,1]$. We denote the \emph{derivatives} of $\Phi_{t}$ with respect to $x$ at $(t,x)$ by $D_{x}\Phi_{t}$.\par

For a flow $\Phi(t,x)$ on $\mathbb{R}^{n}$, the \emph{orbit of x} is the set $O(x)=\{\Phi_{t}(x):t\in \mathbb{R}\}$. The \emph{$\omega$-limit set} of $x$ is $\omega(x)=\cap_{s\geq0}\overline{\cup_{t\geq s}\Phi_{t}(x)}$. We say that the flow $\Phi_{t}$ is \emph{$\omega$-compact} in a subset $X\subset\mathbb{R}^{n}$ if $O(x)$ is relatively compact for each $x\in X$ and $\bigcup_{x\in X}\omega(x)$ is relatively compact in $\mathbb{R}^{n}$. A point $x$ is an \emph{equilibrium} if $O(x)=\{x\}$. The set of equilibria is denoted by $\mathcal{E}$.\par

\begin{defn}\label{kcone}
Let $C\subset\mathbb{R}^{n}$ be a closed set. It is called a \emph{k-cone} of $\mathbb{R}^{n}$ if it satisfies:\par
(i) For any $v\in C$ and $l\in \mathbb{R}$, $lv\in C$;\par
(ii) $\max\{{\rm dim}\ W: W\subset C \text{ is a linear subspace}\}=k$.\\
Moreover, the integer $k(\geq1)$ is called the rank of $C$.
\end{defn}

A $k$-cone $C\subset\mathbb{R}^{n}$ is called \emph{$k$-solid} if there is a $k$-dimensional linear subspace $W$ such that $W\backslash\{0\}\subset \rm{Int}$$C$, and is called \emph{complemented} if there is a $k$-codimensional space $H^{c}\subset\mathbb{R}^{n}$ such that $H^{c}\cap C=\{0\}$. A pair $x,y\in \mathbb{R}^{n}$ are said to be \emph{ordered} if $x-y\in C$, denoted by $x\sim y$; otherwise, $x$ and $y$ are called to be \emph{unordered}. We also write $x\approx y$ and call $x$ and $y$ are \emph{strongly ordered} if $x-y\in \rm{Int}$$C$. Two sets $U,V\subset \mathbb{R}^{n}$ are said to be \emph{strongly ordered} if $x-y\in C$ for any $x\in U$ and $y\in V$, denoted by $U\approx V$.\par

Let $\Phi(t,x)$ denote a flow on $\mathbb{R}^{n}$. We call $\Phi(t,x)$ \emph{monotone with respect to a k-solid cone C} provided $\Phi_{t}(x)\sim\Phi_{t}(y)$, whenever $x\sim y$ and $t\geq0$; and \emph{strongly monotone with respect to a k-solid cone C} provided $\Phi_{t}(x)\approx\Phi_{t}(y)$, whenever $x\neq y$, $x\sim y$ and $t>0$. We call a nontrivial orbit $O(x)$ \emph{pseudo-ordered}, if there are two distinct points $\Phi_{t_{1}}(x)$, $\Phi_{t_{2}}(x)$ in $O(x)$ such that $\Phi_{t_{1}}(x)\sim\Phi_{t_{2}}(x)$; otherwise, $O(x)$ is called \emph{unordered}.
\vskip 2mm

In the following, we give the definition of \emph{prevalence} (see, e.g. \cite{BV10,HSY93,GMH08}). Let $\mathbb{B}$ be a Banach space. A Borel set $W\subset \mathbb{B}$ is \emph{shy} if there exists a nonzero compactly supported Borel measure $\mu$ on $\mathbb{B}$ such that $\mu(v+W)=0$ for every $v\in\mathbb{B}$. More generally, a set is called \emph{shy} if it is contained in a shy Borel set. A set is \emph{prevalent} if its complement is shy. Given $A\subset \mathbb{B}$, we say that a set $W$ is \emph{prevalent in} $A$ if $A\setminus W$ is shy.\par
In general situation, shy sets have the following properties: (i) Every subset of a shy set is shy; (ii) Every translation of a shy set is shy; (iii) No nonempty open set is shy. (iv) Every countable union of shy sets is shy (\cite{HSY93,BV10}). In finite-dimensional spaces, a set $M$ is shy if and only if it has Lebesgue measure zero (see, e.g. \cite[Proposition 2.5]{BV10} or \cite[Fact 6]{HSY93}).\par

\section{The Probe Lemma}\label{sec-DL}

Before we focus on our main theorems, we need to present a technical lemma, called the {\it Probe Lemma}, which turns out to be very crucial for our approach. To simplify the notation, we write $$ S=\{x\in \mathbb{R}^n: \omega(x) \text{ is a singleton}\}.$$
Clearly, $S$ is a Borel set (see \cite[lemma 9]{GMH08}).\par
We call a $k$-dim linear subspace $P\subset\mathbb{R}^{n}$ a \emph{$k$-probe} if $P\subset {\rm Int}C$ (motivated by \cite{BV10}). Given any $x\in \mathbb{R}^{n}$, let $P_{x}=x+P$. Clearly, $P_{x}=P_{y}$ if and only if $x-y\in P$. Let $B(x,\varepsilon)$ be the $\varepsilon$-neighborhood of $x$. The set $B^{P}(x,\varepsilon)=P_{x}\cap B(x,\varepsilon)$ is called the probe $\varepsilon$-neighborhood of $x$.\par

\vskip 2mm

\begin{lem}[\textbf{Probe Lemma}]\label{probe}
Let all hypotheses in Theorem A hold. For any $x\in D\setminus (Q\cup S)$ and any $k$-probe $P$, there exists $\varepsilon>0$ such that $B^{P}(x,\varepsilon)\backslash\{x\}\subset Q$.
\end{lem}

\vskip 2mm

\begin{rmk}
When the rank $k=1$, the classical Monotone Convergence Criterion (c.f. \cite[Theorem 1.2.1]{H95}) yields that $Q\subset S$ in $D$. Under such spacial circumstances, the Probe Lemma naturally reduces to ``{\it For $x\in D\setminus S$, there is a neighborhood $\mathcal{N}$ of $x$ such that, for any $y\in \mathcal{N}$ with $y\sim x$, one has $y\in S$"}. This fact is a part of Sequential Limit Trichotomy (see, e.g. \cite[Theorem 1.4.1(b) or Proposition 1.4.4(d)]{H95}) for classical monotone flows, or parallel to the Discontinuity Principle (see, e.g. \cite[Proposition 3.3]{P92}) for classical monotone discrete-time systems.
\end{rmk}

\vskip 2mm

Before proving the Probe Lemma, we are going to introduce some tools as $k$-Exponential Separation and $k$-Lyapunov Exponents, which we inherit from Feng, Wang and Wu \cite{FWW19} and play an important role in the following proof.\par

Let $K\subset\mathbb{R}^{n}$ be an invariant compact subset. Then, $(\Phi_{t},D\Phi_{t})$ naturally defines a linear skew-product flow on $K\times\mathbb{R}^{n}$. Let $\{Z_{x}\}_{x\in K}$ be a family of $k$-dimensional subspaces of $\mathbb{R}^{n}$, and $G(k,\mathbb{R}^{n})$ consists of $k$-dim linear subspaces of $\mathbb{R}^{n}$ which is a complete metric space by endowing the gap metric (see, e.g. \cite{LW15,LL10}). Then, $K\times(Z_{x})$ is said to be a \emph{$k$-dimensional continuous vector bundle} if the map $K\rightarrow G(k,\mathbb{R}^{n}): x\mapsto Z_{x}$ is continuous. Also, we say the continuous vector bundle $K\times(Z_{x})$ is invariant with respect to $(\Phi_{t},D\Phi_{t})$ if $D\Phi_{t}Z_{x}=Z_{\Phi_{t}x}$ for any $x\in K$ and $t\geq0$.\par

We call $(\Phi_{t},D\Phi_{t})$ admits a $k$\emph{-exponential separation along} $K$ (for short, $k$\emph{-exponential separation}), if there are $k$-dimensional continuous bundle $K\times(E_{x})$ and $(n-k)$-dimensional continuous bundle $K\times(F_{x})$ such that\par
(i) $\mathbb{R}^{n}=E_{x}\oplus F_{x}$, for any $x\in K$;\par
(ii) $D_{x}\Phi_{t}E_{x}=E_{\Phi_{t}(x)}$, $D_{x}\Phi_{t}F_{x}\subset F_{\Phi_{t}(x)}$ for any $x\in K$ and $t>0$;\par
(iii) there are constants $M>0$ and $0<\gamma<1$ such that $$\|D_{x}\Phi_{t}w\|\leq M\gamma^{t}\|D_{x}\Phi_{t}v\|$$ for all $x\in K$, $w\in F_{x}\cap S$, $v\in E_{x}\cap S$ and $t\geq0$, where $S=\{v\in \mathbb{R}^{n}:\|v\|=1\}$.\par
In addition, if $C\subset \mathbb{R}^{n}$ is a $k$-solid cone and\par
(iv) $E_{x}\subset \text{Int}C\cup\{0\}$ and $F_{x}\cap C=\{0\}$ for any $x\in K$, \\
then $(\Phi_{t},D\Phi_{t})$ is called to admit a $k$\emph{-exponential separation along} $K$ \emph{associated with} $C$.\par
It is known that $(\Phi_{t},D\Phi_{t})$ admits a $k$-exponential separation along $K$ associated with $C$ if the flow $\Phi_{t}$ satisfied our assumption (H) (see, e.g. \cite[Corollary 2.2 or Theorem 4.1]{T94} and \cite[Proposition A.1]{M94}).

Let $K\subset\mathbb{R}^{n}$ be an invariant compact subset for $\Phi_{t}$, and $(\Phi_{t},D\Phi_{t})$ admit a $k$-exponential separation $$K\times\mathbb{R}^{n}=E\oplus F,$$ where $E=K\times(E_{x})$ and $F=K\times(F_{x})$. For each $x\in K$, we define the $k$\emph{-Lyapunov exponent} as $\lambda_{kx}=\limsup_{t\rightarrow\infty}\frac{\log m(D_{x}\Phi_{t}|_{E_{x}})}{t},$ where $m(D_{x}\Phi_{t}|_{E_{x}})=\inf_{v\in E_{x}\cap S}\|D_{x}\Phi_{t}v\|$ is the infimum norm of $D_{x}\Phi_{t}$ restricted to $E_{x}$. When $k=1$, $\lambda_{kx}$ naturally reduces to first (or principal) Lyapunov exponent (see, e.g. \cite{PT91}).\par
A \emph{regular point} is a point $x\in K$ for which $\lambda_{kx}=\lim_{t\rightarrow\infty}\dfrac{\log m(D_{x}\Phi_{t}|_{E_{x}})}{t}$.\par

\vskip 2mm
\begin{prop}\label{class}
For any $x\in D$, one of the alternatives must occur:\par
\rm{(i)} $x\in Q\cup S$;\par
\rm{(ii)} For any sequence $\{x_{n}\}_{n=1}^{\infty}\subset D\backslash \{x\}$ with $x_{n}\rightarrow x$ and $x_{n}\sim x$, there are two points $p\in \omega(x)$, $q\in \omega_{c}(x)$ such that $p\sim q$, where $\omega_{c}(x)=\cap_{n\geq1}\overline{\cup_{m\geq n}\omega(x_{m})}$.
\end{prop}

\begin{proof}
Our proof is motivated by the approach in \cite[Section 4]{FWW19}.
Denote by $\omega_{0}(x)$ the set of regular points on $\omega(x)$, i.e., $$\omega_{0}(x)=\{z\in \omega(x): z\text{ is a regular point}\}.$$
Then one of the following alternatives must occur:\par
\rm{(a)} $\lambda_{kz}\leq0$ for some point $z\in \omega_{0}(x)$;\par
\rm{(b)} $\lambda_{kz}>0$ for any $z\in\omega(x)$;\par
\rm{(c)} $\lambda_{kz}>0$ for any $z\in \omega_{0}(x)$; while $\lambda_{k\widetilde{z}}\leq0$ for some $\widetilde{z}\in \omega(x)\backslash\omega_{0}(x)$.\\
We will discuss the three cases one by one.\par

If (a) holds, then it follows from \cite[Theorem 4.2]{FWW19} that $x\in Q\cup S$. Thus, item (i) holds.\par

If (b) holds, we choose any sequence $\{x_{n}\}_{n=1}^{\infty}\subset D\backslash \{x\}$ satisfying $x_{n}\rightarrow x$ and $x_{n}\sim x$. Clearly, $\omega_{c}(x)$ is nonempty, compact and invariant. By virtue of \cite[Lemma 4.4]{FWW19}, there exists a constant $\delta>0$ such that
$$\limsup_{t\rightarrow\infty}\|\Phi_{t}(x_{n})-\Phi_{t}(x)\|\geq\delta, \text{\ \ for each } n\geq1.$$
Consequently, for each $n\geq1$, there are $p_{n}\in \omega(x)$ and $q_{n}\in \omega(x_{n})$ such that $\|p_{n}-q_{n}\|\geq\delta$ and $p_{n}\sim q_{n}$. Without loss of generality, one may assume that $p_{n}\rightarrow p\in \omega(x)$ and $q_{n}\rightarrow q\in \omega_{c}(x)$, as $n\rightarrow \infty$. Therefore, one has $p\sim q$ and $p\neq q$. Thus, item (ii) holds.\par

Finally, if (c) holds, then for any $y\in \omega(x)$, we write $v_{y}:=\frac{d}{dt}|_{t=0}\Phi_{t}(y)$ and let $$\lambda(z,v_{z})=\limsup_{t\rightarrow\infty}\frac{\log\|v_{\Phi_{t}(z)}\|}{t} \text{,\ \ for any } z\in \omega(x)\backslash\omega_{0}(x).$$
Fix any $z\in \omega(x)\backslash \omega_{0}(x)$, clearly, $\lambda(z,v_{z})\leq0$, since $\|v_{\Phi_{t}(z)}\|$ is bounded uniformly for $t\geq0$.\par

When $\lambda(z,v_{z})=0$, we have $v_{z}\notin F_{z}$ (For otherwise, \cite[Lemma 3.6(i)]{FWW19} implies that $\lambda(z,v_{z})\leq \lambda_{kz}+\log(\gamma)< 0$, a contradiction). Therefore, by the $k$-exponential separation, there is $T>0$ such that $D_{z}\Phi_{t}(v_{z})\in \text{Int}C$ for any $t>T$, which implies $z\in Q$. Noticing that $z\in \omega(x)$, we obtain $x\in Q$. Thus, item (i) holds.\par

When $\lambda(z,v_{z})<0$, we have $v_{\Phi_{t}(z)}\rightarrow0$ as $t\to \infty$. This implies $\omega(z)$ only consists of equilibria. Thus, $\omega(z)\subset \omega_{0}(x)$. Hence, due to (c), one has
\begin{equation}\label{3.1}
\lambda_{k\widetilde{z}}>0 \text{,\ \ for any } \widetilde{z}\in \omega(z). \tag{3.1}
\end{equation}
Now, define $C_{x}=\{y\in D: y\neq x \text{ and } y\sim x\}$ and choose a sequence $t_{n}\rightarrow \infty$ such that $\Phi_{t_{n}}(x)\rightarrow z$ as $n\rightarrow \infty$.\par

If there exists $y\in C_{x}$ with a subsequence $\{t_{n_{i}}\}_{i=1}^{\infty}$ of $\{t_{n}\}_{n=1}^{\infty}$, such that $\Phi_{t_{n_{i}}}(y)\rightarrow z$, then \cite[Lemma 4.3]{FWW17} (or \cite[Lemma 2.4]{FWW19}) implies that $z\in Q\cup \mathcal{E}$. Recall that $z\notin \omega_{0}(x)$. Then, we have $z\in Q$. Hence, we again obtain that $x\in Q$, which means item (i) holds.

On the other hand, if for any $a\in C_{x}$, there is a subsequence $t_{n_{j}}^{a}\rightarrow\infty$ of $\{t_{n}\}_{n=1}^{\infty}$ such that $\Phi_{t_{n_{j}}^{a}}(a)\rightarrow z_{a}\neq z$ as $j\rightarrow\infty$. Clearly, $z_{a}$ and $z$ are ordered for any $a\in C_{x}$. By virtue of (\ref{3.1}), we can again utilize \cite[Lemma 4.4]{FWW19} (for $\omega(z)$) to obtain that there exists $\delta>0$ such that
\begin{equation}\label{3.2}
\limsup_{t\rightarrow\infty}\|\Phi_{t}(z_{a})-\Phi_{t}(z)\|\geq\delta, \text{\ \ for any }a\in C_{x}.\tag{3.2}
\end{equation}
Take any sequence $\{x_{k}\}_{k=1}^{\infty}\subset D\backslash \{x\}$ with $x_{k}\rightarrow x$ and $x_{k}\sim x$. For each $k$, we utilize (\ref{3.2}) (hence, for each $z_{x_{k}}$) to obtain that there are $p_{k}\in \omega(z)$ and $q_{k}\in \omega(z_{x_{k}})$ such that $\|p_{k}-q_{k}\|\geq\delta$ and $p_{k}\sim q_{k}$. Recall that $\Phi_{t_{n_{j}}^{x_{k}}}(x_{k})\rightarrow z_{x_{k}}$ as $j\rightarrow\infty$, one has $z_{x_{k}}\in\omega(x_{k})$ and hence $\omega(z_{x_{k}})\subset\omega(x_{k})$. Thus we obtain $q_{k}\in \omega(x_{k})$. And also $p_{k}\in \omega(x)$ since $\omega(z)\subset\omega(x)$. Without loss of generality, one may assume that $p_{k}\rightarrow p\in \omega(x)$ and $q_{k}\rightarrow q\in \omega_{c}(x)$, as $k\rightarrow \infty$. Therefore, one has $p\sim q$ and $p\neq q$, which means item (ii) holds. Thus, we have completed the proof.
\end{proof}

Now, we are going to prove the Lemma \ref{probe}.

\vskip 3mm

\noindent\emph{Proof of lemma \ref{probe}} (\emph{Probe Lemma}).
We prove this lemma by contradiction. Suppose that there exists a sequence $x_n\rightarrow x$ with $x_{n}\neq x$, such that  $x_n\in P_x\setminus Q$  for any $n\in \mathbb{N}$.\par

Then, by recalling that $x\notin Q\cup S$, it follows from Proposition \ref{class} that there are $p\in \omega(x)$ and $q\in \cap_{n\geq1}\overline{\cup_{m\geq n}\omega(x_{m})}$ such that $p\ne q$ and $p\sim q$. Since $\Phi$ is strongly monotone, we may assume that $p\approx q$. Let $U$ and $V$ be the neighborhood of $p$ and $q$, respectively, such that $U\approx V$. Choose two sequences $s_{k}\rightarrow\infty$ and $n_{k}\rightarrow\infty$ such that $\Phi_{s_{k}}(x_{n_{k}})\in V$ for all $k$ sufficiently large. Let also $\tau>0$ be such that $\Phi_{\tau}(x)\in U$. Then, $\Phi_{\tau}(x_{n_{k}})\in U$; and hence, $\Phi_{\tau}(x_{n_{k}})\approx\Phi_{s_{k}}(x_{n_{k}})$ for all $k$ sufficiently large. As a consequence, one has $x_{n_{k}}\in Q$ for all $k$ sufficiently large, a contradiction. Thus, we have completed the proof.
\hfill$\square$

\noindent\section{Proof of the main Theorems}\label{sec4}

In this section, we will prove our main Theorems. For this purpose, besides the Lemma \ref{probe} (Probe Lemma), we still need the following lemma to verify the shyness of a subset $W\subset\mathbb{R}^{n}$.

\begin{lem}\label{shy}
Let $W\subset\mathbb{R}^{n}$ be a Borel subset. Assume that there exists a $k$-probe $P\subset \rm{int}$$C$ such that $W\cap P_{v}$ is at most countable for any $v\in \mathbb{R}^{n}$. Then $W$ is shy.
\end{lem}
\begin{proof}
Fix the $k$-dimensional unit disc $D\subset P$. Define the measure on $\mathbb{R}^{n}$ as
\begin{equation}\label{*}
\mu_{D}(A)\triangleq m(A\cap D),\text{\ \  for any }A\in \mathfrak{B}(\mathbb{R}^{n}). \tag{4.1}
\end{equation}
Here $m$ is the Lebesgue measure on $P$, and $\mathfrak{B}(\mathbb{R}^{n})$ consists of Borel sets of $\mathbb{R}^{n}$ . Clearly, $\mu$ is a nonzero compactly supported Borel measure on $\mathbb{R}^{n}$. Moreover, for any $v\in \mathbb{R}^{n}$, one has
$$(v+W)\cap D\subset(v+W)\cap P=v+(W\cap P_{-v}).$$
Since $W\cap P_{-v}$ is at most countable, we have $m((v+W)\cap D)=0$. By (\ref{*}), this entails that $\mu_{D}(v+W)=0$. Therefore, $W$ is shy.
\end{proof}

Now we are ready to prove the main Theorems.

\vskip 3mm

\noindent$\emph{Proof of Theorem A}.$
Let $N=D\backslash(Q\cup S)$. Then $N$ is a Borel set, since $Q$ and $S$ are both Borel sets. Fix a probe $P\subset \rm{Int}$$C$. In order to show that $N$ is shy, by Lemma \ref{shy}, it suffices to show that $N\cap P_{v}$ is at most countable, for any $v\in \mathbb{R}^{n}$.\par
To this end, fix any $x\in N\cap P_{v}$, then lemma \ref{probe} (Probe Lemma) implies that there exists $\varepsilon>0$ such that $B^{P}(x,\varepsilon)\backslash\{x\}\subset Q$. Consequently, $x$ can not be an accumulation point in $N\cap P_{v}$. By the arbitrariness of $x$ and the separability of $P_{v}$, we obtain that $N\cap P_{v}$ is at most countable.\par
Thus, we have proved that $N$ is shy. Hence, $Q\cup S$ is prevalent in $D$. Finally, in $\mathbb{R}^{n}$, a subset is prevalent if and only if it has full Lebesgue measure. Thus, $Q\cup S$ is full Lebesgue measure in $D$. We have completed the proof.
\hfill$\square$

\vspace{2 ex}

\noindent$\emph{Proof of Theorem B}.$
By virtue of Theorem A, the subset $Q\cup S$ is full Lebesgue measure in $D$. Recall that $k=2$. Then, for any $x\in D\cap(Q\cup S)$, the fact $\omega(x)\cap E=\emptyset$ directly implies that $x\in Q$. It then follows from Feng et al. \cite[Theorem C]{FWW17} (see also Sanchez\cite[Theorem 1]{S09}) that $\omega(x)$ must be a periodic orbit. Thus, we have completed the proof.
\hfill$\square$

\vspace{4 ex}

\noindent\section{Applications}\label{sec6}

In this section, we will apply the obtained Almost Sure Poincar\'e-Bendixson Theorem for high-dimensional autonomous systems to the ordinary differential equations.\par

Consider a system of ODEs
\begin{eqnarray}\label{E1}
\dot{x}=F(x),\ \ \ x\in \mathbb{R}^{n},
\end{eqnarray}
for which $F$ is a $C^{2}$-smooth vector filed defined in $\mathbb{R}^{n}$. Let $\Phi_{t}$ be the flow generated by the system (\ref{E1}). We call the system (\ref{E1}) is \emph{dissipative} if there exists a bounded set $\Lambda\subset\mathbb{R}^{n}$ satisfying for each $x\in \mathbb{R}^{n}$ there exists a constant $t_{0}>0$ such that $\Phi_{t}(x)\in \Lambda$ for all $t\geq t_{0}$.\par
Let $C\subset\mathbb{R}^{n}$ be a $k$-cone. We call the system (\ref{E1}) is $C$-cooperative if the fundamental solution matrix $X^{ij}(t)$ of the linear system
$$\dot{X}=Q^{ij}(t)X,\ \ \ U(0)=I$$
satisfies
\begin{equation}\label{5.2}
Q^{ij}(t)(C\backslash\{0\})\subset \rm{Int}\emph{C},\ \text{for all}\ t>0, \tag{5.2}
\end{equation}
which $Q^{ij}(t)=\int_{0}^{1}DF(\tau\Phi_{t}(i)+(1-\tau)\Phi_{t}(j))d\tau.$\par
Specially, a system satisfying (\ref{5.2}) is called strongly cooperative if $C$ is a classical cone.\par
Assume now that $C\subset\mathbb{R}^{n}$ is a complemented $2$-solid cone. Then, we give the following Almost Sure Poincar\'e-Bendixson Theorem for system (\ref{E1}).
\begin{thm}\label{E}
Assume that system \rm{(}\ref{E1}\rm{)} is dissipative and $C$-cooperative. Then, for almost every point $x\in \mathbb{R}^{n}$, if the $\omega$-limit set $\omega(x)$ contains no equilibrium, then $\omega(x)$ is a periodic orbit.
\end{thm}
\begin{proof}
The flow $\Phi_{t}$ satisfies the assumption (H) if system (\ref{E1}) is $C^{1,\alpha}$-smooth and $C$-cooperative by \cite[Proposition 1]{S09}.\par
Let $B_{i}=\{x\in \mathbb{R}^{n}: \|x\|< i\}$ for any integer $i\geq1$. Since the system (\ref{E1}) is dissipative, the orbit set $O(B_{i})$ of $B_{i}$ is bounded. Let $\mu$ be the Lebesgue measure on $\mathbb{R}^{n}$. According to the Almost Sure Poincar\'e-Bendixson Theorem, for each $B_{i}$, there is a subset $A_{i}\subset B_{i}$ with $\mu(A_{i})=\mu(B_{i})$, such that for any $x\in A_{i}$ the $\omega$-limit set containing no equilibria is a closed orbit.\par
Now, let $A=\cup_{i=1}^{\infty} A_{i}$. Note that $\mathbb{R}^{n}=\cup_{i=1}^{\infty} B_{i}$. We have $\mu(A^{c})=0$ as $A^{c}=\cup_{i\geq1} (B_{i}\backslash A_{i})$. Also, for any $x\in A$, the $\omega$-limit set $\omega(x)$ containing no equilibria is a single closed orbit. We have completed the proof.
\end{proof}

In the rest of this section, we specifically describe a quadratic cone which is 2-solid and complemented. Let $Q$ be a constant non-singular $n\times n$ matrix which is real symmetric. If $Q$ has 2 negative eigenvalues and $(n-2)$ positive eigenvalues, then $C^{-}(Q)=\{x\in \mathbb{R}^{n}: x^{\ast} Q x\leq0\}$ is a complemented 2-solid cone, which $x^{\ast}$ denote the transpose of the vector $x$.\par
In \cite{S09,S10}, Sanchez shows a sufficient conditions for $C^{-}(Q)$-cooperative system. Let $\lambda:\mathbb{R}^{n}\rightarrow \mathbb{R}$ be a continuous function such that the matrices
\begin{equation}\label{5.3}
Q\cdot DF(x)+(DF(x))^{\ast}\cdot Q+\lambda(x)Q<0,\ \text{for any}\ x\in \mathbb{R}^{n},\tag{5.3}
\end{equation}
where $(DF(x))^{\ast}$ is the transpose of the Jacobian $DF(x)$ and ``$<$" represents the usual order in the space of symmetric matrices. Then system (\ref{E1}) is $C^{-}(Q)$-cooperative if it satisfies (\ref{5.3}) (see more details in Sanchez \cite[Proposition 7]{S09}).\par
Along this, we obtain the following Almost Sure Poincar\'e-Bendixson Theorem for high-dimensional systems (\ref{E1}) with a quadratic cone:
\begin{cor}
Assume that system \emph{(}\ref{E1}\emph{)} is dissipative and satisfies \rm{(}\ref{5.3}\rm{)}. Then, for almost every point $x\in \mathbb{R}^{n}$, if the $\omega$-limit set $\omega(x)$ contains no equilibrium, then $\omega(x)$ is a periodic orbit.
\end{cor}
\begin{rmk}
Recently, Feng et.al \cite{FWW19} proved a generic Poincar\'e-Bendixson Theorem for this system (\ref{E1}) with the assumption (\ref{5.3}). We here obtain an almost-sure Poincar\'e-Bendixson Theorem for $n$-dimensional ODE system (\ref{E1}) with a quadratic cone. Thus, we conclude, for $n$-dimensional ODE system (\ref{E1}) with a quadratic cone, the ``Poincar\'e-Bendixson Theorem" holds not only generically, but also almost-surely in measure-theoretic sense.
\end{rmk}

\end{document}